\documentclass[a4paper]{amsart}
\usepackage[english]{babel}
\usepackage[utf8x]{inputenc}
\usepackage[T1]{fontenc}
\usepackage{amsthm}
\usepackage{dsfont}
\usepackage{amssymb}
\usepackage{amsmath}
\theoremstyle{plain}
\usepackage{chngcntr}
\usepackage{relsize}
\usepackage{pdfpages}

\newtheorem{Lemma}{Lemma}
\newtheorem{Theorem}[Lemma]{Theorem}
\newtheorem{Proposition}[Lemma]{Proposition}

\newtheorem{Conjecture}[Lemma]{Conjecture}

\usepackage[a4paper,top=3cm,bottom=2cm,left=3cm,right=3cm,marginparwidth=1.75cm]{geometry}

\usepackage{float}
\usepackage{nccmath}
\usepackage{mathtools}
\usepackage{amsfonts}
\usepackage{amsmath}
\usepackage{graphicx}
\usepackage[colorinlistoftodos]{todonotes}
\usepackage[colorlinks=true, allcolors=blue]{hyperref}
\makeatletter
\newcommand*{\rom}[1]{\expandafter\@slowromancap\romannumeral #1@}
\makeatother

\title{Small solutions of generic ternary quadratic congruences to general moduli}



\author{Stephan Baier}
\address{Stephan Baier,
Ramakrishna Mission Vivekananda Educational and Research Institute, Department of Mathematics, G. T. Road, PO Belur Math, Howrah, West Bengal 711202, India}
\email{stephanbaier2017@gmail.com}
\author{Aishik Chattopadhyay}
\address{Aishik Chattopadhyay,
Ramakrishna Mission Vivekananda Educational and Research Institute, Department of Mathematics, G. T. Road, PO Belur Math, Howrah, West Bengal 711202, India}
\email{aishik.ch@gmail.com}
\begin{document}
\maketitle
\begin{abstract} We study small non-trivial solutions of quadratic congruences of the form $x_1^2+\alpha_2x_2^2+\alpha_3x_3^2\equiv 0 \bmod{q}$, with $q$ being an odd natural number, in an average sense. This extends previous work of the authors in which they considered the case of prime power moduli $q$. Above, $\alpha_2$ is arbitrary but fixed and $\alpha_3$ is variable, and we assume that $(\alpha_2\alpha_3,q)=1$. We show that for all $\alpha_3$ modulo $q$ which are coprime to $q$ except for a small number of $\alpha_3$'s, an asymptotic formula for the number of solutions $(x_1,x_2,x_3)$ to the congruence $x_1^2+\alpha_2x_2^2+\alpha_3x_3^2\equiv 0 \bmod{q}$ with $\max\{|x_1|,|x_2|,|x_3|\}\le N$ and $(x_3,q)=1$ holds if $N\ge q^{11/24+\varepsilon}$ and $q$ is large enough. It is of significance that we break the barrier 1/2 in the above exponent. Key tools in our work are Burgess's estimate for character sums over short intervals and Heath-Brown's estimate for character sums with binary quadratic forms over small regions whose proofs depend on the Riemann hypothesis for curves over finite fields. We also formulate a refined conjecture about the size of the smallest solution of a ternary quadratic congruence, using information about the Diophantine properties of its coefficients. 
\end{abstract}

\tableofcontents

\section{Introduction and main result}
Throughout this paper, assume that $\varepsilon$ is a fixed but arbitrarily small positive number. All implied $O$-constants will be allowed to depend on $\varepsilon$. 

The study of small solutions of quadratic congruences $Q(x_1,...,x_n)\equiv 0 \bmod{q}$,
$Q$ being an integral quadratic form, has attracted a lot of attention. In this paper, we study small solutions of generic ternary diagonal forms. We will justify the term "generic" below. If $(x_1,x_2,x_3)\in \mathbb{Z}^3$ is a solution to the above congruence, we call the quantity
$\max\{|x_1|,|x_2|,|x_3|\}$ "height" of this solution.   

If $q$ is odd and squarefree, Heath-Brown \cite[Theorem 2]{HB} proved that for any integral ternary quadratic form $Q(x_1,x_2,x_3)$ with determinant coprime to $q$, there exists a non-trivial solution $(x_1,x_2,x_3)\in \mathbb{Z}^3$ to the congruence
\begin{equation} \label{congruence}
Q(x_1,x_2,x_3)\equiv 0\bmod{q}
\end{equation}
of height $\ll q^{5/8+\varepsilon}$. (Here "non-trivial" means that $(x_1,x_2,x_3)\not=(0,0,0)$.) He conjectured that there should be a non-trivial solution of height $\ll q^{1/2+\varepsilon}$. By a result of Cochrane (see \cite{Coc}), this is true if $Q$ has {\it fixed} coefficients and $q$ tends to infinity, where the implied constant may depend on the form.  (In fact, Cochrane established this for an exponent of $1/2$ in place of $1/2+\varepsilon$ and all moduli $q$.)  Throughout the sequel, we keep the condition $(\det Q,q)=1$. This condition is important as Heath-Brown gave examples of forms with $(\det Q,q)>1$ for which there is no non-trivial solution of height $\ll q^{2/3-\varepsilon}$.

As pointed out in \cite{HB}, it is easy to extend Heath-Brown's above result from odd and squarefree to all odd moduli $q$, as the following argument shows. Write $q=q_0q_1^2$, where $q_0$ is squarefree. Then by the above result, there is a non-trivial solution to the congruence 
$$
Q(x_1,x_2,x_3)\equiv 0\bmod{q_0}
$$
of height $\ll q_0^{5/8+\varepsilon}$. This extends to a solution $(x_1q_1,x_2q_1,x_3q_1)$ to the congruence \eqref{congruence} of height $\ll q_0^{5/8+\varepsilon}q_1\ll q^{5/8+\varepsilon}$.
An extreme case is that of a large power of an odd prime $p$: By the above argument, we see that for $q=p^n$, there is a non-trivial solution to \eqref{congruence} of height
$\ll_p q^{1/2+\varepsilon}$.
However, if we put suitable restrictions on the variables, such as coprimality to the modulus $q$, then the above simple argument is no longer applicable. In \cite{snu}, Haldar and the first-named author proved for diagonal ternary forms and odd prime power moduli $q=p^n$ that there exists a solution to \eqref{congruence} satisfying $(x_1x_2x_3,q)=1$ of height
$\ll_p q^{11/18+\varepsilon}$.  

It seems reasonable to conjecture that \eqref{congruence} has always a solution satisfying the coprimality condition $(x_1x_2x_3,q)=1$ of height
$\ll q^{1/2+\varepsilon}$. This has been established in \cite[Theorem 1]{snu} for {\it fixed} coefficients $\alpha_i$ and $q$ tending to infinity over the powers of a fixed odd prime $p>5$. (In fact, the authors proved an asymptotic formula for the number of solutions of height $N\ge q^{1/2+\varepsilon}$.) It does not diminish the interest of this conjecture if we demand coprimality to $q$ of just one variable instead of all three - the above extension argument from square-free to arbitrary odd moduli still breaks down in this case. Indeed, in our main result below we will just assume that $(x_3,q)=1$. This will facilitate our calculations. 

We are not aware of any literature in which a significant improvement of the above conjecture has been attempted. Indeed, it is true that in general, the exponent $1/2$ in this conjecture cannot be reduced. For example, the congruence
$$
x_1^2+x_2^2+x_3^2\equiv 0 \bmod{q}
$$ 
has no non-trivial solution of height less than $\sqrt{q/3}$. However, for the case of odd prime power moduli, we proved in \cite{GTQC} that, in a sense, almost all diagonal ternary forms admit a solution satisfying $(x_3,q)=1$ of significantly smaller height $\ll q^{11/24+\varepsilon}$ ($\ll q^{1/3+\varepsilon}$ under the Lindel\"of hypothesis for Dirichlet $L$-functions).  Casually speaking, such a solution exists for {\it generic} ternary diagonal forms.  In this paper, we extend this result to all odd moduli $q$. Precisely, we prove the following.

\begin{Theorem} \label{mainthm}
Let $q\in \mathbb{N}$ be odd and $\alpha_1,\alpha_2\in \mathbb{Z}$ such that $(\alpha_1\alpha_2,q)=1$. Then for all 
$$\alpha_3\in \Phi(q):=\{s\in \mathbb{Z}: 1\le s\le q, \ (s,q)=1\}
$$ 
with at most $o(\varphi(q))$ exceptions, the congruence
$$
\alpha_1x_1^2+\alpha_2x_2^2+\alpha_3x_3^2\equiv 0 \bmod{q}
$$ 
has a solution $(x_1,x_2,x_3)\in \mathbb{Z}^3$ satisfying $(x_3,q)=1$ of height $\ll q^{11/24+\varepsilon}$. More precisely, if $q^{11/24+\varepsilon}\le N\le q$, then for all  $\alpha_3\in \Phi(q)$ with at most $o(\varphi(q))$ exceptions, the number of solutions of height less or equal $N$ satisfies the
asymptotic formula
\begin{equation} \label{asympform}
\sum\limits_{\substack{|x_1|,|x_2|,|x_3|\le N\\ (x_3,q)=1\\ 
x_1^2+\alpha_2x_2^2+\alpha_3x_3^2\equiv 0 \bmod{q}}} 1=C_q\cdot \frac{(2N)^3}{q}\cdot \left(1+o(1)\right),
\end{equation}
where
$$
C_q:=\prod\limits_{p|q}\left(1-\frac{1}{p}\right)\cdot \prod\limits_{p|q}\left(1-\frac{1}{p}\cdot \left(\frac{-\alpha_2}{p}\right)\right).
$$
Here $\left(\frac{\cdot}{p}\right)$ denotes the Legendre symbol. 
Moreover, under the Lindel\"of hypothesis for Dirichlet $L$-functions, the exponent $11/24$ above can be replaced by  $1/3$. 
\end{Theorem}  

We see that the exponent $1/3$ cannot be reduced: If $N$ is much smaller than $q^{1/3}$, then the right-hand side of \eqref{asympform} is much smaller than 1, and thus we cannot expect any solutions in this case. 

The result in Theorem \ref{mainthm} raises the question if the above conjecture can be improved using information on the coefficients $\alpha_1,\alpha_2,\alpha_3$.  We will give a heuristic suggesting the following refined conjecture, taking into account Diophantine properties of the fractions $\alpha_i/q$. 

\begin{Conjecture} \label{conj} Let $q\in \mathbb{N}$ be odd and $\alpha_1,\alpha_2,\alpha_3\in \mathbb{Z}$ such that $(\alpha_1\alpha_2\alpha_3,q)=1$. Then the congruence
$$
\alpha_1x_1^2+\alpha_2x_2^2+\alpha_3x_3^2\equiv 0 \bmod{q}
$$ 
has a solution $(x_1,x_2,x_3)\in \mathbb{Z}^3$ satisfying $(x_1x_2x_3,q)=1$ of height
\begin{equation}\label{Nineq}
\ll q^{\varepsilon}\max\bigg\{q^{1/3},\max\limits_{\substack{r\bmod q\\r\not \equiv 0\bmod{q}}}\min\left\{\left|\left|r\alpha_1/q\right|\right|^{-1/2},\left|\left|r\alpha_2/q\right|\right|^{-1/2},\left|\left|r\alpha_3/q\right|\right|^{-1/2}\right\}\bigg\}.
\end{equation}
\end{Conjecture}

In a nutshell, this conjecture tells us that the smallest solution should be of height $\ll q^{1/3+\varepsilon}$ unless $\alpha_1/q$, $\alpha_2/q$, $\alpha_3/q$ have good simultaneous approximation by fractions $a_1/r$, $a_2/r$, $a_3/r$ with a small denominator $r$, respectively (see our discussion in section \ref{heu} for details).  For example, if $\alpha_1=\alpha_2=\alpha_3=1$, then we may take $r=1$ and $a_1=a_2=a_3=0$ to recover the exponent $1/2+\varepsilon$.  
 
If we wish, we can extend Conjecture \ref{conj} to arbitrary ternary forms $Q$ with $(\det Q,q)=1$: In this case, the roles of $\alpha_1$, $\alpha_2$, $\alpha_3$ are taken by the eigenvalues of the matrix corresponding to  $Q$ modulo $q$. \\ \\
{\bf Acknowledgements.} The authors would like to thank the Ramakrishna Mission Vivekananda Educational and Research Insititute for an excellent work environment. The research of the second-named author was supported by a CSIR Ph.D fellowship under file number 09/0934(13170)/2022-EMR-I. 

\section{Preliminaries}
Our key tools are estimates for short character sums of the form
$$
S_1=\sum\limits_{|x|\le N} \chi(x)  \quad \mbox{and} \quad S_2=\sum\limits_{|x_1|,|x_2|\le N} \chi(Q(x_1,x_2)), 
$$
where $\chi$ is a non-principal Dirichlet character, $Q(x_1,x_2)$ is a binary quadratic form and $N$ is small compared to the modulus of $\chi$.  To this end, we use results by Burgess and Heath-Brown whose proofs rely on the Riemann hypothesis for curves over finite fields. However, we have to extend them from primitive to non-principal characters, which in the case of the sum $S_2$ takes some efforts. 
Below are the character sum estimates used in this paper.

\begin{Proposition}\label{Burgess}
Let $M\ge 0$, $N\in \mathbb{N}$ and $\chi$ be a non-principal Dirichlet character modulo $q>1$. Then
$$
\sum\limits_{M<n\le M+N} \chi(n) \ll_{r} N^{1-1/r}q^{(r+1)/(4r^2)+\varepsilon}
$$ 
for $r=2,3$, and for any $r\in \mathbb{N}$ if $q$ is cube-free. 
\end{Proposition}

\begin{proof}
For primitive characters, this result is due to Burgess (see \cite[Theorem 12.6]{IwKo}, for example). Now let 
$\chi$ be a general non-principal character modulo $q>1$. Assume that $q=q_1q_2$, where $q_1>1$ is the conductor of $\chi$, so that $\chi=\chi_1\chi_2$, where $\chi_1$ is the primitive character modulo $q_1$ inducing $\chi$ and  $\chi_2$ is the principal character modulo $q_2$. Then it follows that 
\begin{equation*}
\begin{split}
\sum\limits_{M<n\le M+N} \chi(n)= & \sum\limits_{M<n\le M+N} \chi_1(n)\chi_2(n)\\
= & \sum\limits_{\substack{M<n\le M+N\\ (n,q_2)=1}}  \chi_1(n) \\
= & \sum\limits_{d|q_2} \mu(d)\sum\limits_{\substack{M<n\le M+N\\ d|n}}  \chi_1(n)\\
= & \sum\limits_{d|q_2} \mu(d)\chi_1(d)\sum\limits_{M/d<n\le M/d+N/d}  \chi_1(n)
\end{split}
\end{equation*}
using M\"obius inversion.
Now applying Burgess's estimate to the inner-most sum involving the primitive character $\chi_1$, we deduce that
\begin{equation*}
\begin{split}
\sum\limits_{d|q_2} \mu(d)\chi_1(d)\sum\limits_{M/d<n\le M/d+N/d}  \chi_1(n)\ll_{r} & \sum\limits_{d|q_2}  \left(\frac{N}{d}\right)^{1-1/r}q_1^{(r+1)/(4r^2)+\varepsilon}\\ 
\ll &
 N^{1-1/r}q^{(r+1)/(4r^2)+2\varepsilon}.
\end{split}
\end{equation*}
Hence, the claimed estimate follows upon redefining $\varepsilon$. 
\end{proof}

Under the Lindel\"of Hypothesis for Dirichlet $L$-functions, we have the following sharper estimate for the case when $M=0$.

\begin{Proposition}\label{Linde}
Let $N\in \mathbb{N}$ and $\chi$ be a non-principal Dirichlet character modulo $q>1$ which is induced by a primitive character $\chi_1$. Then
$$
\sum\limits_{0<n\le N} \chi(n) \ll  N^{1/2}q^{\varepsilon},
$$ 
provided that $L(1/2+it,\chi_1)\ll (|t|q)^{\varepsilon}$ whenever $|t|\ge 1$.  
\end{Proposition}

\begin{proof}Similarly as in the proof of Proposition \ref{Burgess}, we first reduce the sum in question to a sum involving the primitive character $\chi_1$. The result then follows in a standard way by using Perron's formula and contour integration, shifting the line of integration to $\Re s=1/2$.  
\end{proof}

\begin{Proposition}\label{HeathBrown}
Let an integer $r\ge 3$ be given, and suppose that $C\subset \mathbb{R}^2$ is a convex set contained in a disc $\{{\bf x} \in \mathbb{R}^2:
 ||{\bf x}-{\bf x}_0||_2 \le R\}$, $||.||_2$ denoting the Euclidean norm. Let $q_0 \ge 2$ be odd and squarefree, and let $\chi$ be a non-principal character modulo $q_0$ with conductor $q_1>1$. Then if $Q(x,y)$ is a binary integral quadratic form with $(\det(Q), q) = 1$, we have
\begin{equation} \label{smallR}
\sum\limits_{(x,y)\in C}  \chi(Q(x, y)) \ll_{r} R^{2-1/r}q_1^{(r+2)/(4r^2)}q_0^{\varepsilon} \quad 
\mbox{ if } q_1^{1/4+1/(2r)} \le R \le q_1^{5/12+1/(2r)}
\end{equation}
and 
\begin{equation} \label{largeR}
\sum\limits_{(x,y)\in C}  \chi(Q(x, y)) \ll \left(R^{5/3}q_1^{5/36}+R^2q_1^{-1/18}\right)q_0^{\varepsilon} \quad  \mbox{ if } R> q_1^{7/12}. 
\end{equation}
\end{Proposition}

\begin{proof}
For primitive characters, the estimate \eqref{smallR} was established by Heath-Brown in \cite{HB}[Theorem 3]. We will extend his arguments in \cite{HB} to cover general non-principal characters and the range $R>q_1^{7/12}$ in the appendix. 
\end{proof}

We will also use the following well-known results about quadratic Gauss sums. 

\begin{Proposition} \label{Gauss sums}
Let $c$ be odd and squarefree, and assume that $(a,c)=1$. Set  
\begin{equation} \label{Gaussdef}
G(a,c):=\sum\limits_{n=1}^c e\left(\frac{an^2}{c}\right).
\end{equation}
Then 
\begin{equation} \label{Gaussev}
G(a,c)=\left(\frac{a}{c}\right)\cdot \epsilon_c\sqrt{c},
\end{equation}
where 
\begin{equation} \label{tauev}
\epsilon_c=\begin{cases} 
      1 & \text{if } c\equiv 1 \bmod{4} \\
      i & \text{if } c\equiv 3 \bmod{4}.
\end{cases}
\end{equation}
Moreover, for all $n\in \mathbb{Z}$, we have the relation
\begin{equation} \label{relation}
\left(\frac{n}{c}\right)=\frac{1}{\epsilon_c \sqrt{c}}\cdot \sum_{k=1}^{c}\left(\frac{k}{c}\right)e\left(\frac{nk}{c}\right).
\end{equation} 
\end{Proposition}

\begin{proof} These properties of quadratic Gauss sums can be found in \cite{BEW}[Chapter 1], for example.
\end{proof}

\section{Initial approach}
We first observe that we may assume without loss of generality that $\alpha_1=1$ in Theorem \ref{mainthm} since otherwise, we may divide our quadratic congruence by $\alpha_1$. Now
our initial approach is very similar to that in \cite{GTQC}. We literally copy several steps.

Suppose that the conditions in Theorem \ref{mainthm} are satisfied and $\alpha_1=1$.  Set 
$$
S(\alpha_3):=\sum\limits_{\substack{|x_1|,|x_2|,|x_3|\le N\\ (x_3,q)=1\\ 
x_1^2+\alpha_2x_2^2+\alpha_3x_3^2\equiv 0 \bmod{q}}} 1.
$$
We detect the congruence condition
$$
x_1^2+\alpha_2x_2^2+\alpha_3x_3^2\equiv 0 \bmod{q}
$$
via orthogonality relations for Dirichlet characters. Recalling the condition $(\alpha_3x_3,q)=1$, we have 
$$
\frac{1}{\varphi(q)}
\sum\limits_{\chi \bmod q} \chi\left(x_1^2+\alpha_2x_2^2\right)\overline{\chi}\left(-\alpha_3 x_3^2\right) = \begin{cases} 1 & \mbox{ if } x_1^2+\alpha_2x_2^2+\alpha_3x_3^2\equiv 0 \bmod{q}\\ 0 & \mbox{ if }x_1^2+\alpha_2x_2^2+\alpha_3x_3^2\not\equiv 0 \bmod{q}. \end{cases}
$$
It follows that 
$$
S(\alpha_3)= \frac{1}{\varphi(q)} \sum\limits_{\chi \bmod{q}}\
\sum_{|x_1|,|x_2|,|x_3|\le N}  \chi\left(x_1^2+\alpha_2x_2^2\right)\overline{\chi}\left(-\alpha_3 x_3^2\right).
$$
The main term contribution comes from the principal character $\chi_0 \bmod{q}$. Thus we may split the above into
\begin{equation} \label{splitting}
S(\alpha_3)=M+E(\alpha_3),
\end{equation}
where 
\begin{equation} \label{maintermdef}
M:=\frac{1}{\varphi(q)} 
\sum_{\substack{|x_1|,|x_2|,|x_3|\le N\\ \left(x_1^2+\alpha_2x_2^2,q\right)=1\\  (x_3,q)=1}} 1
\end{equation}
is the main term and 
\begin{equation} \label{Edef}
E(\alpha_3):=\frac{1}{\varphi(q)} \sum\limits_{\substack{\chi \bmod{q}\\ \chi\not=\chi_0}}\
\sum_{|x_1|,|x_2|,|x_3|\le N}  \chi\left(x_1^2+\alpha_2x_2^2\right)\overline{\chi}\left(-\alpha_3 x_3^2\right)
\end{equation}
is the error term. The main term will be evaluated in the next section. 

To derive Theorem \ref{mainthm}, we will estimate the variance 
\begin{equation} \label{Vdefi}
V:=\sum\limits_{\substack{\alpha_3=1\\ (\alpha_3,q)=1}}^q \left|S(\alpha_3)-M\right|^2=\sum\limits_{\substack{\alpha_3=1\\ (\alpha_3,q)=1}}^q \left|E(\alpha_3)\right|^2.
\end{equation}
Our goal is to beat the estimate $O\left(N^6q^{-1}\right)$ in order to deduce that for almost all $\alpha_3 \bmod{q}$ with $(\alpha_3,q)=1$, the size of the error term $E(\alpha_3)$ is smaller than that of the main term $M$. Plugging in the right-hand side of \eqref{Edef} for $E(\alpha_3)$ and using orthogonality relations for Dirichlet characters, we have 
\begin{equation*}
\begin{split}
V=&\frac{1}{\varphi(q)^2} \sum\limits_{\alpha_3=1}^q \bigg|\sum\limits_{\substack{\chi \bmod{q}\\ \chi\neq\chi_0}}\overline{\chi}(-\alpha_3)\sum\limits_{|x_1|,|x_2|\leq N} \chi\left(x_1^2+\alpha_2 x_2^2\right)\sum\limits_{|x_3|\leq N}\overline{\chi}^2(x_3)\bigg|^2\\
=&\frac{1}{\varphi(q)^2}\sum\limits_{\substack{\chi_1,\chi_2\bmod{q}\\ \chi_1,\chi_2\neq \chi_0}}\ \sum\limits_{\alpha_3=1}^q \overline{\chi_1}\chi_2(-\alpha_3)\sum\limits_{|x_1|,|x_2|\leq N}\chi_1\left(x_1^2+\alpha_2x_2^2\right)\sum\limits_{|y_1|,|y_2|\leq N} \overline{\chi_2}\left(y_1^2+\alpha_2y_2^2\right)\times\\ & \sum\limits_{|x_3|\leq N} \overline{\chi_1}^2(x_3)  \sum\limits_{|y_3|\leq N} \chi_2^2(y_3)\\
=&\frac{1}{\varphi(q)}\sum\limits_{\substack{\chi \bmod{q}\\ \chi\neq\chi_0}}\bigg|\sum\limits_{|x_1|,|x_2|\leq N} \chi\left(x_1^2+\alpha_2x_2^2\right)\sum_{|x_3|\leq N}\overline{\chi}^2(x_3)\bigg|^2.
\end{split}
\end{equation*}

Next, we separate the summation into two parts: the contributions of characters $\chi$ with $\chi^2=\chi_0$ and $\chi^2\not=\chi_0$, respectively. We note that the only characters modulo $q$ of order two are of the form 
$$
\chi(x)=\left(\frac{x}{q_1}\right)\chi_2(x),
$$ 
where $q_1>1$, $q_1q_2=\mbox{rad}(q)$ is the largest squarefree divisor of $q$ (the radical of $q$),  $\left(\frac{x}{q_1}\right)$ is the Jacobi symbol, and $\chi_2$ is the principal character modulo $q_2$. To see this, note that these characters are indeed of order two, there are $2^{\omega(q)}-1$ characters of this form, and the number of elements of order two in $(\mathbb{Z}/q\mathbb{Z})^{\ast}$ is $2^{\omega(q)}-1$ as well  (recall that the 
character group modulo $q$ is isomorphic to $(\mathbb{Z}/q\mathbb{Z})^{\ast}$). The latter is a consequence of the Chinese remainder theorem and Hensel's lemma. Hence, these are the only characters modulo $q$ of order two. Consequently, we obtain
\begin{equation} \label{Vsplit}
V=V_1+V_2,
\end{equation}
where 
\begin{equation}\label{s1}
    V_1:=\frac{1}{\varphi(q)} \cdot  \sum\limits_{\substack{q_1|\text{rad}(q)\\ q_1>1}} \bigg|\sum_{\substack{|x_1|,|x_2|\leq N\\ (x_1^2+\alpha_2x_2^2,q_2)=1}} \left(\frac{x_1^2+\alpha_2x_2^2}{q_1}\right)\bigg|^2\cdot \bigg|\sum\limits_{\substack{|x_3|\le N\\ (x_3,p)=1}} 1\bigg|^2 
\end{equation}
with $q_1q_2=\mbox{rad}(q)$, and 
\begin{equation}\label{s2}
    V_2=\frac{1}{\varphi(q)} \sum\limits_{\substack{\chi\bmod{q}\\ \chi^2\neq \chi_0}}\bigg|\sum_{|x_1|,|x_2|\leq N} \chi\left(x_1^2+\alpha_2x_2^2\right) \bigg|^2 \cdot \bigg| \sum_{|x_3|\leq N} \overline{\chi}^2(x_3)\bigg|^2.
\end{equation}    

\section{Approximation of the main term}
In this section, we approximate the main term $M$, defined in \eqref{maintermdef}. We begin by writing 
\begin{equation*} 
M=\frac{1}{\varphi(q)}\cdot KL,
\end{equation*}
where 
$$
K:=\sum_{\substack{|x_1|,|x_2|\le N\\ \left(x_1^2+\alpha_2x_2^2,q\right)=1}} 1
$$
and 
$$
L:=\sum_{\substack{|x_3|\le N\\ (x_3,q)=1}} 1.
$$
Using M\"obius inversion and the bound $\tau(n)\ll n^{\varepsilon}$ for the divisor function, the term $L$ above can be approximated by
\begin{equation*} 
L=\sum_{d|q}\mu(d)\sum_{\substack{|x_3|\le N\\ d|x_3}} 1=\sum_{d|q}\mu(d)\left(\frac{2N}{d}+O(1)\right)=2N\cdot \frac{\varphi(q)}{q}+O\left(q^{\varepsilon}\right).
\end{equation*}

Similarly, we use M\"obius inversion to write the term $K$ above as
\begin{align*}
K:=\sum_{d|q}\mu(d)\sum\limits_{\substack{|x_1|,|x_2|\leq N\\d|(x_1^2+\alpha_2 x_2^2)}}1.
\end{align*}
If $d$ is squarefree and $(d,x_1)=e$, then $d|(x_1^2+\alpha_2 x_2^2)$ is equivalent to $e|x_2$ and $(d/e)|(x_1^2+\alpha_2 x_2^2)$. Hence, splitting, the right-hand side above into subsums according to the greatest common divisor of  $d$ and $x_1$, we get
\begin{align*}
K= \sum_{d|q}\mu(d)\sum_{e|d}\sum\limits_{\substack{|x_1|\leq N\\(d,x_1)=e}}\sum\limits_{\substack{|x_2|\leq N\\e|x_2\\ x_1^2+\alpha_2 x_2^2\equiv 0\bmod{d/e}}}1.
\end{align*}
Writing $x_1=y_1e$ and $x_2=y_2e$ and using the fact that $(d/e,e)=1$ if $d$ is squarefree, it follows that
\begin{align*}
K= &\sum_{d|q}\mu(d)\sum_{e|d}\sum\limits_{\substack{|y_1|\leq N/e\\(d/e,y_1)=1}}\sum\limits_{\substack{|y_2|\leq N/e\\ y_1^2+\alpha_2 y_2^2\equiv 0\bmod{d/e} }}1.
\end{align*}
Using the Chinese remainder theorem, for any given $y_1$ coprime to $d/e$, the total number of solutions $y_2$ of the congruence above equals 
$\prod_{p|(d/e)}\left(1+\left(\frac{-\alpha_2}{p}\right)\right)$. Consequently,
\begin{align*}
K =&\sum_{d|q}\mu(d)\sum_{e|d}\Bigg(\sum\limits_{\substack{|y_1|\leq N/e\\(d/e,y_1)=1}} 1\Bigg)\prod_{p|(d/e)}\left(1+\left(\frac{-\alpha_2}{p}\right)\right)\left(\frac{2N}{d}+O(1)\right)\\
=&\sum_{d|q}\mu(d)\left(\frac{2N}{d}+O(1)\right)\sum_{e|d}\left(\frac{2N}{e}\cdot \frac{\varphi(d/e)}{d/e}+O(1)\right) \prod_{p|(d/e)}\left(1+\left(\frac{-\alpha_2}{p}\right)\right)\\
=& (2N)^2 \sum_{d|q}\frac{\mu(d)}{d^2} \cdot \sum_{e|d} \varphi(d/e)\cdot  \prod_{p|(d/e)}\left(1+\left(\frac{-\alpha_2}{p}\right)\right)+O\left(Nq^{\varepsilon}\right)\\
= & (2N)^2 \sum_{d|q}\frac{\mu(d)}{d^2} \cdot \sum_{f|d} \varphi(f)\cdot  \prod_{p|f}\left(1+\left(\frac{-\alpha_2}{p}\right)\right)+O\left(Nq^{\varepsilon}\right),
\end{align*}
where we have estimated the sum over $y_1$ in a similar way as the term $L$ above. Rewriting the sums over $f$ and $d$ as products, we obtain
\begin{align*}
  \sum_{d|q}\frac{\mu(d)}{d^2} \cdot \sum_{f|d} \varphi(f)\cdot  \prod_{p|f}\left(1+\left(\frac{-\alpha_2}{p}\right)\right)  
   =& \sum_{d|q}\frac{\mu(d)}{d^2}\cdot \prod_{p|d}\left(1+\varphi(p)\left(1+\left(\frac{-\alpha_2}{p}\right)\right)\right)\\
   =& \prod_{p|q}\left(1-\frac{1}{p^2}\cdot \left(1+\varphi(p)\left(1+\left(\frac{-\alpha_2}{p}\right)\right)\right)\right)\\
   =&\prod_{p|q}\left(1-\frac{1}{p}\right)\cdot \prod_{p|q}\left(1-\frac{1}{p}\cdot \left(\frac{-\alpha_2}{p}\right)\right)=:C_q.
\end{align*}
Combining everything in this section, we arrive at
\begin{equation} \label{maintermappro}
M=C_q\cdot \frac{(2N)^3}{q}+O\left(\frac{N^2}{q^{1-\varepsilon}}\right).
\end{equation}

\section{Estimation of $V_2$}
Our treatment of $V_2$, defined in \eqref{s2}, is literally the same as in \cite{GTQC}. We copy it here. First, we note that
\begin{equation} \label{V2ini}
V_2\le \frac{1}{\varphi(q)} \sum\limits_{\chi \bmod q}\bigg|\sum_{|x_1|,|x_2|\leq N} \chi\left(x_1^2+\alpha_2x_2^2\right) \bigg|^2 \cdot \max\limits_{\substack{\chi\bmod{q}\\ \chi\not=\chi_0}} \bigg| \sum_{|x_3|\leq N} \chi(x_3)\bigg|^2.
\end{equation}
Expanding the modulus square, and using orthogonality relations for Dirichlet characters, the sum over $\chi$ above transforms into 
\begin{equation} \label{double1}
\begin{split}
& \sum_{\chi\bmod q}\bigg| \sum_{|x_1|,|x_2|\leq N} \chi\left(x_1^2+\alpha_2 x_2^2\right)\bigg|^2\\
=&\sum_{\chi \bmod q}\ \sum\limits_{|x_1|,|x_2|,|y_1|,|y_2|\leq N} \chi\left(x_1^2+\alpha_2x_2^2\right)\overline{\chi}\left(y_1^2+\alpha_2y_2^2\right)\\
=&\varphi(q) \sum\limits_{\substack{|x_1|,|x_2|,|y_1|,|y_2|\leq N\\ (x_1^2+\alpha_2x_2^2,q)=1\\ (y_1^2+\alpha_2y_2^2,q)=1\\ x_1^2+\alpha_2x_2^2\equiv y_1^2+\alpha_2y_2^2 \bmod q}} 1.
\end{split}
\end{equation}
Furthermore, under the conditions $(\alpha_2,q)=1$ and $N<q/2$, we have 
\begin{equation} \label{double2}
\begin{split}
& \sum\limits_{\substack{|x_1|,|x_2|,|y_1|,|y_2|\leq N\\
x_1^2+\alpha_2x_2^2\equiv y_1^2+\alpha_2y_2^2\bmod q}}1\\
=& \sum\limits_{\substack{|x_1|,|x_2|,|y_1|,|y_2|\leq N \\ (x_1-y_1)(x_1+y_1)\equiv\alpha_2(y_2-x_2)(y_2+x_2) \bmod q}}1 \\
= &\sum\limits_{\substack{|x_1|,|x_2|,|y_1|,|y_2|\leq N\\x_1=\pm y_1\text{ and }x_2=\pm y_2}}1+\sum\limits_{\substack{0<|k_1|,|k_2|\leq 4N^2\\ k_1\equiv\alpha_2k_2 \bmod q}}\ \sum\limits_{\substack{|x_1|,|x_2|,|y_1|,|y_2|\leq N\\ (x_1-y_1)(x_1+y_1)=k_1\\(y_2-x_2)(y_2+x_2)=k_2}}1\\
\ll& N^2 + \sum_{0<|k_2|\leq 4N^2}\sum\limits_{\substack{0<|k_1|\leq 4N^2\\k_1\equiv \alpha_2k_2\bmod q}} \tau(|k_1|)\tau(|k_2|)\\
\ll & N^{2+\varepsilon}\left(1+\frac{N^2}{q}\right),
\end{split}
\end{equation}
where we use the bound $\tau(n)\ll_{\varepsilon} n^{\varepsilon}$ for the divisor function. From \eqref{double1} and \eqref{double2}, we obtain
\begin{equation} \label{double}
\sum_{\chi\bmod q}\bigg| \sum_{|x_1|,|x_2|\leq N} \chi\left(x_1^2+\alpha_2 x_2^2\right)\bigg|^2\\
\ll qN^{2+\varepsilon}\left(1+\frac{N^2}{q}\right).
\end{equation}

Applying Propositions \ref{Burgess} with $r=2$ and Proposition \ref{Linde}, we get
\begin{equation}\label{bb}
\max\limits_{\substack{\chi\bmod{q}\\ \chi\not=\chi_0}} \bigg|\sum_{|x_3|\leq N}\overline{\chi}(x_3)\bigg|^2 =\begin{cases} O\left(Nq^{3/8+\varepsilon}\right) \mbox{ unconditionally,}\\ \\
O\left(Nq^{\varepsilon}\right) \mbox{ under the Lindel\"of hypothesis.}
\end{cases}
\end{equation}
Combining \eqref{V2ini}, \eqref{double} and \eqref{bb}, we find that 
$$
V_2= \begin{cases} O\left(\left(1+N^2q^{-1}\right)N^3q^{3/8+\varepsilon}\right) \mbox{ unconditionally,}\\ \\
O\left(\left(1+N^2q^{-1}\right)N^3q^{\varepsilon}\right) \mbox{ under the Lindel\"of hypothesis.}
\end{cases}
$$   

We aim to achieve a bound of the form
\begin{equation} \label{desiredV2}
V_2\ll \Delta N^6q^{-1},
\end{equation}
where $\Delta$ is small compared to 1.  To this end, we observe that for $(u,v)\in \mathbb{R}^2$ with $u<6$ and $\Delta\in (0,1)$,
\begin{equation*}
N^uq^v\le \Delta N^6q^{-1}\Longleftrightarrow N\ge \Delta^{-1/(6-u)} q^{(v+1)/(6-u)}.
\end{equation*}
It follows that 
\begin{equation} \label{V2esti}
V_2=\begin{cases} O\left(\Delta N^6q^{-1}\right) \mbox{ if } N\ge q^{\varepsilon}\max\left\{\Delta^{-1/3}q^{11/24},\Delta^{-1}q^{3/8}\right\} \mbox{ unconditionally,}\\ \\
O\left(\Delta N^6q^{-1}\right) \mbox{ if } N\ge q^{\varepsilon}\max\left\{\Delta^{-1/3}q^{1/3},\Delta^{-1}\right\}\mbox{ under the Lindel\"of hypothesis.}
\end{cases}
\end{equation}

\section{Estimation of $V_1$}
In this section, we estimate the term $V_1$, defined in \eqref{s1}. Let $q_0:=\mbox{rad}(q)$. Denote the character sum over $x_1$ and $x_2$ on the right-hand side of \eqref{s1} by 
$$
L(q_1):=\sum\limits_{\substack{|x_1|,|x_2|\leq N\\(x_1^2+\alpha_2x_2^2,q_2)=1}}\left(\frac{x_1^2+\alpha_2x_2^2}{q_1}\right).
$$ 
When $2N\le q_1^{7/12}$, we apply \eqref{smallR} in  Proposition \ref{HeathBrown} with  
$C:=\{(x_1,x_2)\in \mathbb{R}^2 : \max\{|x_1|,|x_2|\}\le N\}$, $x_0:=0$, $R:=2N$, $\chi(x):=\left(\frac{x}{q_1}\right)\chi_2(x)$ and $Q(x_1,x_2):=x_1^2+\alpha_2x_2^2$ to bound this sum by
\begin{equation} \label{V1esti1}
L(q_1)\ll_{r}  N^{2-1/r}q_1^{(r+2)/(4r^2)}q^{\varepsilon}
\quad \mbox{ if } 
q_1^{1/4+1/(2r)} \le 2N \le q_1^{5/12+1/(2r)},
\end{equation}
where $r\ge 3$ is a suitable integer. When $2N>q_1^{7/12}$ and $q_1$ is not too small, we use \eqref{largeR} in Proposition \ref{HeathBrown} to bound this sum by 
 \begin{equation} \label{V1esti1'}
\sum\limits_{(x,y)\in C}  \chi(Q(x, y)) \ll \left(N^{5/3}q_1^{5/36}+N^2q_1^{-1/18}\right)q^{\varepsilon}.
\end{equation}
When $q_1$ is very small, we estimate $L(q_1)$ via a direct completion argument, worked out below.  

Using M\"obius inversion, we have 
\begin{equation} \label{Mobi} 
L(q_1)= \sum\limits_{d|q_2} \mu(d) \sum\limits_{\substack{|x_1|,|x_2|\leq N\\ d|(x_1^2+\alpha_2x_2^2)}}\left(\frac{x_1^2+\alpha_2x_2^2}{q_1}\right)=L^{\sharp}(q_1)+L^{\flat}(q_1),
\end{equation}
where $L^{\sharp}(q_1)$ is the contribution of $d\le N$ and $L^{\flat}(q_1)$ is the remaining contribution of $d>N$. We bound $L^{\flat}(q_1)$ by 
\begin{equation} \label{larged}
|L^{\flat}(q_1)|\le \sum\limits_{\substack{d|q_2\\ d>N}} 
\sum\limits_{|x_2|\leq N} \sum\limits_{\substack{x_1\bmod{d}\\ x_1^2\equiv -\alpha_2x_2^2\bmod{d}}}1 \ll   \sum\limits_{d|q_2} Nd^{\varepsilon}\ll Nq_2^{2\varepsilon}.
\end{equation}
To bound $L^{\sharp}(q_1)$, we divide the summations over $x_1$ and $x_2$ into residue classes modulo $q_1$ and $d$ and use the Chinese remainder theorem to obtain
\begin{equation} \label{smalld}
\begin{split}
L^{\sharp}(q_1)=& \sum\limits_{\substack{d|q_2\\ d\le N}} \mu(d)  \sum\limits_{a_1,a_2\bmod{q_1}} \left(\frac{a_1^2+\alpha_2a_2^2}{q_1}\right)
\sum\limits_{\substack{b_1,b_2\bmod{d}\\ b_1^2+\alpha_2b_2^2\equiv 0 \bmod{d}}} \
\sum\limits_{\substack{|x_1|,|x_2|\le N\\ x_1\equiv a_1\bmod{q_1}\\ x_1\equiv b_1\bmod{d}\\ x_2\equiv a_2\bmod{q_1}\\ x_2\equiv b_2\bmod{d}}} 1\\
= &  \sum\limits_{\substack{d|q_2\\ d\le N}} \mu(d)  \sum\limits_{a_1,a_2\bmod{q_1}} \left(\frac{a_1^2+\alpha_2a_2^2}{q_1}\right)
\sum\limits_{\substack{b_1,b_2\bmod{d}\\ b_1^2+\alpha_2b_2^2\equiv 0 \bmod{d}}} 
\left(\frac{2N}{q_1d}+O(1)\right)\left(\frac{2N}{q_1d}+O(1)\right)\\
= &  \left(\frac{N}{q_1}\right)^2 \Bigg(\sum\limits_{a_1,a_2\bmod{q_1}} \left(\frac{a_1^2+\alpha_2a_2^2}{q_1}\Bigg)\right)\cdot \Bigg( \sum\limits_{\substack{d|q_2\\ d\le N}} \frac{\mu(d)}{d^2} \cdot \sum\limits_{\substack{b_1,b_2\bmod{d}\\ b_1^2+\alpha_2b_2^2\equiv 0 \bmod{d}}} 1\Bigg)
+O\left(Nq_1^2(Nq_2)^{2\varepsilon}\right),
\end{split}
\end{equation}
where we use the fact that the number of solutions $(b_1,b_2)$ to the congruence $b_1^2+\alpha_2b_2^2\equiv 0 \bmod{d}$ is bounded by $d^{1+\varepsilon}$. Now it is easy to see that 
\begin{equation} \label{Tq1}
T(q_1):=\sum\limits_{a_1,a_2\bmod{q_1}} \left(\frac{a_1^2+\alpha_2a_2^2}{q_1}\right)=0,
\end{equation}
as the following calculation shows: Using the relation \eqref{relation}, we have 
  \begin{equation*}
\begin{split}
  T(q_1)= & \frac{1}{\epsilon_{q_1}\sqrt{q_1}}\cdot \sum\limits_{k=1}^{q_1}\left(\frac{k}{q_1}\right)\sum_{a_1,a_2\bmod q_1}e\left(\frac{k(a_1^2+\alpha_2a_2^2)}{q_1}\right)\\
  = & \frac{1}{\epsilon_{q_1}\sqrt{q_1}}\cdot \sum\limits_{k=1}^{q_1}\left(\frac{k}{q_1}\right)G(k,q_1)G(k\alpha_2,q_1),
\end{split}
\end{equation*}
where $G(a,c)$ is the quadratic Gauss sum, defined in \eqref{Gaussdef}. Using its evaluation in \eqref{Gaussev}, it follows that 
\begin{equation*}
\begin{split}
T_1(q_1)= \left(\frac{\alpha_2}{q_1}\right)\cdot \epsilon_{q_1}\sqrt{q_1}\cdot \sum_{k=1}^{q_1} \left(\frac{k}{q_1}\right)=0.
\end{split}
\end{equation*}
Combining \eqref{Mobi}, \eqref{larged}, \eqref{smalld} and \eqref{Tq1}, we deduce that
\begin{equation} \label{V1esti2}
L(q_1)\ll Nq_1^2q^{\varepsilon}
\end{equation}
if $N\le q$ upon redefining $\varepsilon$.

We aim to establish a bound of the form
\begin{equation} \label{Lq1target}
L(q_1)\ll \Delta^{1/2} N^2q^{-\varepsilon} 
\end{equation}
so that 
\begin{equation} \label{desiredV1}
V_1\ll \Delta N^6 q^{-1}
\end{equation}
using \eqref{s1}.  We observe that for $(u,v)\in \mathbb{R}^2$ with $u<2$ and $\Delta\in (0,1)$,
\begin{equation*}
N^uq_1^vq^{\varepsilon}\le \Delta^{1/2} N^2q^{-\varepsilon}\Longleftrightarrow N\ge \left(q^{2\varepsilon}\Delta^{-1/2}\right)^{1/(2-u)} q_1^{v/(2-u)}.
\end{equation*}
Hence, \eqref{V1esti1}, \eqref{V1esti1'} and \eqref{V1esti2} imply that
\begin{equation} \label{V1combi}
L(q_1)\ll_r \Delta^{1/2} N^2q^{-\varepsilon}\begin{cases} & \mbox{ if } \left(q^{2\varepsilon}\Delta^{-1/2}\right)^{r}q_1^{1/4+1/(2r)}\le 2N\le q_1^{5/12+1/(2r)} \\ & \mbox{ for some } r\in \mathbb{N} \mbox{ with } r\ge 3,\\ \\
&  \mbox{ if } 2N\ge q^{6\varepsilon}\Delta^{-3/2}q_1^{5/12}  \mbox{ and } q_1\ge q^{36\varepsilon}\Delta^{-9},\\ \\
& \mbox{ if } 2N\ge q^{2\varepsilon}\Delta^{-1/2}q_1^2.
\end{cases}
\end{equation}

\section{Proof of Theorem \ref{mainthm}}
In the following, we assume that $\varepsilon$ is small enough and $N\ge q^{1/3+15\varepsilon}$, and we take $\Delta:=q^{-\varepsilon}$, which will suffice to prove Theorem \ref{mainthm}. Then under the Lindel\"of hypothesis, the second bound in \eqref{V2esti} yields the desired estimate \eqref{desiredV2} for $V_2$, and the first bound in \eqref{V2esti} gives this estimate under the stronger condition $N\ge q^{11/24+2\varepsilon}$ on $N$. 

Next, we establish the desired estimate \eqref{desiredV1} for $V_1$, for which we need to prove that the bound \eqref{Lq1target} holds for all $q_1$ dividing $q_0=\mbox{rad}(q)$.  Assume first that $q_1\ge q^{100\varepsilon}$.  Under this condition, it is easily checked that the intervals
$$
\left(q^{2\varepsilon}\Delta^{-1/2}\right)^{r}q_1^{1/4+1/(2r)}\le 2N\le q_1^{5/12+1/(2r)}
$$
are overlapping and cover a range of  
$$
q^{15\varepsilon}q_1^{1/3}\le 2N\le q_1^{7/12}
$$
if $3\le r\le 6$. Also, we have 
$$
q^{6\varepsilon}\Delta^{-3/2}q_1^{5/12}\le q_1^{7/12}\quad \mbox{ and } \quad q_1\ge q^{36\varepsilon}\Delta^{-9}
$$
under the above conditions. Hence, if $q_1\ge q^{100\varepsilon}$, then using the first two estimates in \eqref{V1combi}, we see that the said bound \eqref{Lq1target} holds whenever $N\ge q^{1/3+15\varepsilon}$. 
If $q_1< q^{100\varepsilon}$, then we obtain \eqref{Lq1target} by an application of the third estimate in \eqref{V1combi} whenever $N\ge q^{1/3+15\varepsilon}$ and $\varepsilon$ is small enough. Thus, all ranges are covered and \eqref{desiredV1} holds. Now combining \eqref{Vsplit}, \eqref{desiredV2} and \eqref{desiredV1}, we have
\begin{equation} \label{Vestimate}
V\ll \Delta N^6q^{-1}.
\end{equation}

Recalling that $\Delta:=q^{-\varepsilon}$ with $\varepsilon$ small enough, it follows from \eqref{Vdefi}, \eqref{maintermappro} and \eqref{Vestimate} that
\begin{equation} \label{condi6}
\sum\limits_{\substack{\alpha_3=1\\ (\alpha_3,q)=1}}^q \left|S(\alpha_3)-C_q\cdot \frac{N^3}{q}\right|^2=\begin{cases} O\left(N^6q^{-\varepsilon-1}\right) \mbox{ if } N\ge q^{11/24+2\varepsilon} \mbox{ unconditionally,}\\ \\
O\left(N^6q^{-\varepsilon-1}\right) \mbox{ if } N\ge q^{1/3+15\varepsilon}\mbox{ under the Lindel\"of hypothesis.}
\end{cases}
\end{equation}
Now using $C_q\gg q^{-\varepsilon/8}$, we observe that if the left-hand side of \eqref{condi6} is $O\left(N^6q^{-\varepsilon-1}\right)$, then we have 
$$
S(\alpha_3)=C_q\cdot \frac{N^3}{q}\cdot \left(1+O\left(q^{-\varepsilon/4}\right)\right)
$$
for all 
$$
\alpha_3\in \{s\in \mathbb{Z} : 1\le s\le q, \ (s,q)=1\}
$$ 
with at most $O\left(\varphi(q)q^{-\varepsilon/4}\right)$ exceptions. This together with \eqref{condi6} implies the result of Theorem \ref{mainthm} upon redefining $\varepsilon$.

\section{Heuristic}
Assume that $(\alpha_1\alpha_2\alpha_3,q)=1$. In section 1, we stated the conjecture that the congruence 
\begin{equation} \label{cong}
\alpha_1x_1^2+\alpha_2x_2^2+\alpha_3x_3^2\equiv 0 \bmod{q}
\end{equation} 
should always have a solution satisfying
$(x_1x_2x_3,q)=1$ of height $\ll q^{1/2+\varepsilon}$. However, Theorem \ref{mainthm} indicates that for almost all such congruences, a much stronger bound for the smallest solution $(x_1,x_2,x_3)$ satisfying $(x_1x_2x_3,q)=1$ should hold. (The coprimality condition in Theorem \ref{mainthm} was just $(x_3,q)=1$, but with some extra efforts, a result of the same strength under the stronger condition $(x_1x_2x_3,q)=1$ should be possible to establish.) This raises the question if the said conjecture can be refined by using information on the coefficients $\alpha_1,\alpha_2,\alpha_3$.  In this section, we address this question.

As pointed out in section 1, the particular congruence $x_1^2+x_2^2+x_3^2\equiv 0\bmod{q}$ has no non-trivial solution of height less than $\sqrt{q/3}$ since in this case, this congruence turns into the equation $x_1^2+x_2^2+x_3^2=0$. Similarly, if $\alpha_1,\alpha_2,\alpha_3$ are fixed non-zero integers having the same sign, then there is no non-trivial solution of height $\ll q^{1/2-\varepsilon}$ to the congruence \eqref{cong} if $q$ is large enough. 
Naturally, one may ask under which more general conditions on the coefficients $\alpha_1,\alpha_2,\alpha_3$, a similar argument implies the non-existence of non-trivial solutions. The following considerations demonstrate that this may happen when $\alpha_1/q$, $\alpha_2/q$, $\alpha_3/q$ have good simultaneous approximation by fractions $a_1/r$, $a_2/r$, $a_3/r$ with a small denominator $r$. In the following, we make this precise. (Approximations of this form were actually {\it utilized} in \cite[section 5]{snu}.)  
       
Suppose that $r\in \mathbb{N}$ and 
$$\left|\left|\frac{r\alpha_i}{q}\right|\right|=\beta_i \quad \mbox{ for } i=1,2,3
$$
so that
$$\frac{r\alpha_i}{q}=a_i+\beta_i \quad \text{ for } i=1,2,3 \text{ and some } a_i\in \mathbb{Z}.
$$
Now multiplying the congruence \eqref{cong}
by $r$ gives 
$$r\alpha_1x_1^2+r\alpha_2x_2^2+r\alpha_3x_3^2\equiv 0\bmod{q}
$$
which is the same as 
$$
(a_1q+\beta_1q)x_1^2+(a_2q+\beta_2q)x_2^2+(a_3q+\beta_3q)x_3^2\equiv 0\bmod{q}.
$$
Reducing the left-hand side modulo $q$ gives
\begin{equation} \label{newcong}
\beta_1qx_1^2+\beta_2qx_2^2+\beta_3qx_3^2\equiv 0\bmod q, 
\end{equation}
where $\beta_iq\in \mathbb{Z}$ for $i=1,2,3$. If $r\not\equiv 0 \bmod{q}$, then $\beta_i\not=0$ since $(\alpha_i,q)=1$ for $i=1,2,3$.  Suppose that $|x_1|,|x_2|,|x_3|\leq N$ and
$$\max\{|\beta_1|,|\beta_2|,|\beta_3|\}<\frac{1}{3N^2}.
$$
Then the above congruence \eqref{newcong} turns into an equation
$$
\beta_1qx_1^2+\beta_2qx^2+\beta_3qx_3^2=0.
$$
If $\beta_1,\beta_2,\beta_3$ have the same sign, then this equation does not have any non-trivial solution. In order to expect the existence of a non-trivial solution for large enough $q$, the least we should demand is that
\begin{equation} \label{demand}
\max\{|\beta_1|,|\beta_2|,|\beta_3|\}\geq \frac{q^{2\varepsilon}}{N^2}
\end{equation}
for $\varepsilon>0$ arbitrary but fixed and every choice of $r\not\equiv 0 \bmod{q}$.  
The above inequality \eqref{demand} is equivalent to
\begin{equation}\label{*}
N\geq q^{\varepsilon}\min\{|\beta_1|^{-1/2},|\beta_2|^{-1/2},|\beta_3|^{-1/2}\}.
\end{equation}
We conjecture that the validity of \eqref{*} for all $r\not\equiv 0\bmod{q}$ together with the condition $N\geq q^{1/3+\varepsilon}$ from the conditional part of Theorem \ref{mainthm} suffices to ensure the existence of a solution satisfying $(x_1x_2x_3,q)=1$ to the congruence \eqref{cong} of height at most $N$. This gives rise to Conjecture \ref{conj}.

We note that in the case when $\alpha_1=1$, the inequality \eqref{Nineq} for the height is equivalent to 
$$
\ll q^{\varepsilon}\max\left\{q^{1/3},\max_{1\leq r< q^{1/3}}\min\left\{(q/r)^{1/2},\left|\left|r\alpha_2/q\right|\right|^{-1/2},\left|\left|r\alpha_3/q\right|\right|^{-1/2}\right\}\right\}
$$
and use this to check the consistency of Conjecture \ref{conj} with Theorem \ref{mainthm}. To this end, it suffices to establish that 
\begin{equation} \label{es}
\max_{1\le r<q^{1/3}}\left|\left|r\alpha_3/q\right|\right|^{-1/2}>q^{1/3}
\end{equation}
for at most $o(\varphi(q))$ integers $\alpha_3\in \{1,...,q\}$. Given $r$, we have
\begin{equation}\label{**}
\left|\left|r\alpha_3/q\right|\right|^{-1/2}>q^{1/3}
\end{equation}
if $r\alpha_3\equiv b\bmod{q}$ with $|b|<q/2$ and $\left(q/|b|\right)^{1/2}>q^{1/3}$, which is equivalent to $|b|<q^{1/3}$. Moreover, given $r$ and $b$, there are at most $(r,q)$ solutions $\alpha_3$ to the congruence $r\alpha_3\equiv b\bmod{q}$. Thus for every $r$, there are at most $O((r,q)q^{1/3})$  integers $\alpha_3\in \{1,...,q\}$ such that \eqref{**} holds. Hence, there are at most 
$$
\ll \sum\limits_{1\le r< q^{1/3}} (r,q)q^{1/3}\ll q^{2/3+\varepsilon}
$$ 
integers $\alpha_3\in \{1,...,q\}$ such that \eqref{es} holds. This verifies the claim.

\section{Appendix: Proof of Proposition \ref{HeathBrown}} \label{heu}
We slightly modify the arguments in \cite[section 4]{HB}, where the same result was proved for primitive characters. 
Here we just indicate the required changes but refer the reader to \cite{HB} for computational details. 

Assume that $\chi$ is a non-principal character to a squarefree modulus $q_0$. Let $q_1>1$ be the conductor of $\chi$ and assume that $q_0=q_1q_2$. Then $\chi=\chi_1\chi_2$, where $\chi_1$ is a primitive character modulo $q_1$, and $\chi_2$ is the principal character modulo $q_2$. Using M\"obius inversion, it follows that
\begin{align} \label{view}
 \sum_{(x_1,x_2)\in C}\chi(Q(x_1,x_2))=\sum\limits_{\substack{(x_1,x_2)\in C\\(Q(x_1,x_2),q_2)=1}}\chi_1(Q(x_1,x_2))
 =\sum_{d|q_2}\mu(d) \Sigma_d,
\end{align}
where 
\begin{equation} \label{sigmad}
\Sigma_d:=\sum\limits_{\substack{(x_1,x_2)\in C\\d|Q(x_1,x_2)}}\chi_1(Q(x_1,x_2)).
\end{equation}
The remaining task is to estimate $\Sigma_d$.  To this end, we proceed similarly as in \cite[section 4]{HB}, with the only differences that now an additional condition $d|Q(x_1,x_2)$ is included, and $q$ is replaced by $q_1$ and $\chi$ by $\chi_1$.   

Let $N\in \mathbb{N}$ be a parameter satisfying $N\leq Rq_1^{-1/100}$, to be fixed later. Set $S:=[R/N]$. (Note that we used the symbol $N$ previously with a different meaning.) The method starts with specifying a set of "good" vectors ${\bf s}\in \mathbb{N}^2$. For their definition, we refer the reader to \cite[section 4]{HB}. 
All we use here is that these "good" vectors form a set $\mathcal{S}$ whose cardinality satisfies the lower bound
\begin{equation}\label{SP!}
\sharp \mathcal{S} \gg_{\varepsilon}S^2q_1^{-\varepsilon}\quad \mbox{ if } S\gg q_1^{\varepsilon},
\end{equation}
by a result in \cite{BrHe}, and that $\mathcal{S}$ is constructed in such a way that the inequality
\begin{equation}\label{new}
\Sigma_{d}\ll N^{-1}S^{-2}q_1^{\varepsilon}\sum\limits_{\mathbf{s}\in\mathcal{S}}\sum\limits_{\substack{\mathbf{x}\in \mathbb{Z}^2\\||\mathbf{x-x_0}||_2\leq 2R}}\max_{I\subseteq (0,N]}\Bigg|\sum_{\substack{n\in I\\ d|Q(x_1+ns_1,x_2+ns_2)}} \chi_1(Q(x_1+ns_1,x_2+ns_2))\Bigg|
\end{equation}
holds. (Inequality \eqref{SP!} above is the same as \cite[(4.2)]{HB} with $q$ replaced by $q_1$, and inequality \eqref{new} above is  \cite[inequality for $\Sigma$ before (4.3)]{HB} with $q,\chi$ replaced by $q_1,\chi_1$ and an additional summation condition $d|Q(x_1+ns_1,x_2+ns_2)$ included. This comes from the additional summation condition $d|Q(x_1,x_2)$ in the definition of $\Sigma_d$ in \eqref{sigmad}.)

Proceeding similarly as in \cite[section 4]{HB}, we continue with splitting the quadratic form on the right-hand side of \eqref{new} into
$$
Q(x_1+ns_1,x_2+ns_2)=Q({\bf s})\widetilde{Q}(n+a(\mathbf{s,x}),b(\mathbf{s,x})),
$$
where $\widetilde{Q}(X,Y)$ is a suitable quadratic form and $a(\mathbf{s,x})$, $b(\mathbf{s,x})$ are certain integers depening on ${\bf s}$ and ${\bf x}$ (again, for the details see \cite[section 4]{HB}). Now if $(Q({\bf s}),d)=e$, then $d|Q(x_1+ns_1,x_2+ns_2)$ is equivalent to $f|\widetilde{Q}(n+a(\mathbf{s,x}),b(\mathbf{s,x}))$, where $d=ef$. Hence,
writing 
$$
N(a,b):=\sharp\left\{(\mathbf{s,x})\in\mathcal{S}\times\mathbb{Z}^2: ||\mathbf{x}-\mathbf{x}_0||_2\leq 2R,\ a(\mathbf{s,x})=a,\ b(\mathbf{s,x})=b\right\},
$$
it follows that 
\begin{equation}\label{sumd}
\Sigma_{d}\ll\sum_{f|d}N^{-1}S^{-2}q_1^{\varepsilon}\sum_{a,b \bmod{q_1}}N(a,b)\max_{I\subseteq (0,N]}\Bigg|\sum\limits_{\substack{n\in I\\f|\widetilde{Q}(n+a,b)}}\chi(\widetilde{Q}(n+a,b))\Bigg|.
\end{equation}
Here we have dropped the condition $(Q(\mathbf{s}),d)=e$ while defining $N(a,b)$, which is tolerable. 

Proceeding along the lines in \cite[section 4]{HB}, we now use estimates for the first and second moments of $N(a,b)$, remove the maximum on the right-hand side of \eqref{sumd} and use H\"older's inequality (see the relevant parts of \cite{HB} for the details). This reduces the problem to estimating the $2r$-th moments
$$
S(q_1,H):=\sum_{a,b\bmod{q_1}}\Bigg|\sum\limits_{\substack{n\leq H\\f|\widetilde{Q}(n+a,b)}}\chi_1(\widetilde{Q}(n+a,b))\Bigg|^{2r}.
$$
In contrast to \cite[section 4]{HB}, we here need to handle an additional divisor condition $f|\widetilde{Q}(n+a,b)$. We recall that $f|q_2$ and $(q_1,q_2)=1$, so $(q_1,f)=1$. Now we extend the outer summation over $a,b$ to all residue classes modulo $q_1f$ and write $a\equiv a_1f+a_2q_1\bmod{q_1f}$ and 
$b\equiv b_1f+b_2q_1 \bmod{q_1f}$, where $a_1,b_1$ run over all residue classes modulo $q_1$, and $a_2,b_2$ run over all residue classes modulo $f$. In this way, we obtain
\begin{align*}
S(q_1,H)\le & \sum\limits_{\substack{a_1\bmod q_1\\ b_1\bmod q_1}} \sum\limits_{\substack{a_2 \bmod f\\b_2\bmod f}}\Bigg|\sum\limits_{\substack{n\leq H\\f|\widetilde{Q}(n+a_2q_1,b_2q_1)}}\chi_1(\widetilde{Q}(n+a_1f,b_1f)\Bigg|^{2r}\\
=&\sum\limits_{\substack{c_2,c_3,d_2\bmod{f}\\ \widetilde{Q}(c_3,d_2)\equiv 0\bmod{f}}}
\sum\limits_{c_1,d_1 \bmod{q_1}}\Bigg|\sum\limits_{\substack{n\leq H\\n\equiv c_3-c_2\bmod{f}}}\chi_1(\widetilde{Q}(n+c_1,d_1))\Bigg|^{2r}
\end{align*}
via the changes of variables $c_1\equiv a_1f\bmod{q_1}$, $d_1\equiv b_1f\bmod{q_1}$, $c_2\equiv a_2q_1\bmod{f}$, $c_3\equiv n+c_2\mod f$, $d_2\equiv b_2q_1\bmod{f}$. 
Upon taking $c=c_3-c_2$, it follows that 
$$
S(q_1,H)\ll f^{2+\varepsilon} \max_{c\bmod{f}} \sum\limits_{c_1,d_1 \bmod{q_1}}\Bigg|\sum\limits_{\substack{n\leq H\\n\equiv c\bmod{f}}}\chi_1(\widetilde{Q}(n+c_1,d_1))\Bigg|^{2r}
$$
since the number of solutions $(c_3,d_2)$ to the congruence $\widetilde{Q}(c_3,d_2)\equiv 0\bmod{f}$ is $\ll f^{1+\varepsilon}$. Writing $n=fk+c$, the above implies that
$$
S(q_1,H)\ll f^{2+\varepsilon} \max\limits_{x\in \mathbb{R}} \sum\limits_{u,v \bmod{q_1}}\Bigg|\sum\limits_{k\in (x,x+H/f]} \chi_1(\widetilde{Q}(fk+u,v))\Bigg|^{2r}. 
$$
The sum on the right-hand side can be handled by the same technique as the sum
$$
 \sum\limits_{a,b \bmod{q}}\Bigg|\sum\limits_{n\le H} \chi(\widetilde{Q}(n+a,b))\Bigg|^{2r}
$$ 
in \cite[section 4]{HB} using the Riemann hypothesis for curves over finite fields, where it is important to note that $f$ is coprime to the modulus $q_1$ of $\chi_1$. In this way, we obtain the estimate
$$
S(q_1,H) \ll_r f^2(fq_1H)^{\varepsilon}\left(q_1\left(\frac{H}{f}\right)^{2r}+q_1^2\left(\frac{H}{f}\right)^r\right)
$$
similarly as in  \cite[Lemma 9]{HB}. Since $r\ge 3$, this implies the estimate
$$
S(q_1,H)\ll_r (qH)^{\varepsilon}\left(q_1H^{2r}+q_1^2H^r\right). 
$$
Now, along the same lines as in \cite[section 4]{HB}, we obtain the estimate
$$
\Sigma_{d}^{2r}\ll_r N^{2-2r}R^{4r-4}\left(1+R^2N^{-1}q_1^{-1/2}+R^4N^{-2}q_1^{-4/3}\right)(q_1Nd)^{\varepsilon}\left(q_1N^{2r} +q_1^2N^r\right),
$$
where the divisor sum over $f$ in \eqref{sumd} creates the extra factor $d^{\varepsilon}$.  Similarly as in \cite[end of section 4]{HB}, in order to balance the final two terms, we choose $N:= \left[q_1^{1/r}\right]$, which satisfies
our constraint $N \le Rq_1^{-1/100}$ provided that $R\ge q_1^{1/4+1/(2r)}$ and $r \ge 3$. On
redefining $\varepsilon$,  it follows that
\begin{equation*}
\begin{split}
\Sigma_d^{2r} \ll_r & (q_1d)^{\varepsilon} N^{2-2r}R^{4r-4}\left(1 + R^2N^{-1}q_1^{-1/2} + R^4N^{-2}q_1^{-4/3}\right)q_1N^{2r}\\
\ll_r & (q_1d)^{\varepsilon} q_1^{1/2+1/r}R^{4r-2}\left(R^{-2}q^{1/2+1/r} + 1 + R^2q_1^{-5/6-1/r}\right).
\end{split}
\end{equation*}
Now Proposition \ref{HeathBrown} follows using \eqref{view}, where the estimate \eqref{largeR} arrives on choosing 
$r=3$.


\begin{thebibliography}{20}	
\bibitem{snu} S.Baier, A. Haldar, {\it Asymptotic behavior of small solutions of quadratic congruences in three variables modulo prime powers.} Res. number theory 8, No. 3, Paper No. 58 , 24 p. (2022).

\bibitem{BrHe} T.D. Browning, D.R.Heath-Brown, {\it Rational points on quartic hypersurfaces}, J. Reine Angew. Math. 629, 37--88 (2009).

\bibitem{GTQC} S. Baier, A. Chattopadhyay, {\it Small solution on Generic Ternary Quadratic Congruences}, 
Preprint, arXiv:2406.09778 (2024).

\bibitem{BEW} B.C. Berndt; R.J. Evans; K.S. Williams, {\it 
Gauss and Jacobi sums},
Canadian Mathematical Society Series of Monographs and Advanced Texts. New York, NY: John Wiley \& Sons. xi, 583 p. (1998).

\bibitem{Coc} T. Cochrane, 
{\it On representing the multiple of a number by a quadratic form},
Acta Arith. 63, No. 3, 211--222 (1993).

\bibitem{HB} D. R. Heath-Brown, {\it Small solutions of quadratic congruences, and character sums with binary quadratic forms}, Mathematika 62, No.2, 551--571 (2016).

\bibitem{IwKo} H. Iwaniec, E. Kowalski, {\it Analytic Number Theory}, Colloquium Publications. American Mathematical Society 53. Providence, RI: American Mathematical Society (AMS) xi, 615 p. (2004).

\end{thebibliography}
\end{document}